\documentclass[11pt]{amsart}

\usepackage[T1]{fontenc}      
\usepackage[utf8]{inputenc}

\usepackage[margin=1.1in]{geometry}
\usepackage{amsmath, amssymb, amsthm, mathtools}
\usepackage{microtype}
\usepackage{enumitem}
\usepackage{listings}
\usepackage[hidelinks]{hyperref}

\theoremstyle{plain}
\newtheorem{theorem}{Theorem}
\newtheorem{lemma}[theorem]{Lemma}

\theoremstyle{definition}
\newtheorem{remark}[theorem]{Remark}

\lstdefinestyle{python}{
  language=Python,
  basicstyle=\ttfamily\footnotesize,
  keywordstyle=\bfseries,
  commentstyle=\itshape,
  numbers=left,
  numberstyle=\tiny,
  numbersep=6pt,
  breaklines=true,
  breakatwhitespace=true,
  showstringspaces=false,
  frame=single,
  framesep=4pt,
  tabsize=2,
  xleftmargin=14pt,
  columns=fullflexible,
  upquote=true,
}

\title[Mathar's recurrence for OEIS A176677]%
{A short proof of Mathar's 2016 recurrence conjecture for OEIS A176677}

\author{Tong Niu}
\email{mrnt0810@gmail.com}

\date{\today}


\begin{document}

\maketitle

\begin{abstract}
For the OEIS sequence A176677, defined by the quadratic convolution
recurrence $a(0) = a(1) = 1$ and $a(n+1) = \sum_{p=0}^n a(p) a(n-p) - 1$
for $n \ge 1$, R.~J.~Mathar contributed in March 2016 the conjectured
order-4 P-recursive recurrence
\[
   (n+1)\,a(n) + 2(-3n+1)\,a(n-1) + (9n-13)\,a(n-2)
        - 4\,a(n-3) + 4(-n+4)\,a(n-4) = 0.
\]
We give a short proof. The convolution recurrence translates directly
into the algebraic equation
$z(1-z) G(z)^2 - (1-z) G(z) + (1 - z - z^2) = 0$ for the ordinary
generating function $G(z)$, and Mathar's recurrence then drops out as the
coefficient form of a 1st-order linear inhomogeneous ODE
$q_0(z) G(z) + q_1(z) G'(z) = R(z)$ that we verify by polynomial
division modulo the algebraic equation. The polynomial $q_1(z)$ admits
the factorization $q_1(z) = -z(z-1)(2z-1)(2z^2 + 3z - 1)$,
whose roots are exactly the singularities of $G$. Deutsch's
combinatorial interpretation (Motzkin paths of length $n-1$ with
two-coloured level-zero horizontal steps) is preserved.
\end{abstract}

\section{Introduction}\label{sec:intro}

The On-Line Encyclopedia of Integer Sequences (OEIS,
\url{https://oeis.org}) collects integer sequences whose
``Conjecture: \dots'' comments record formulas, recurrences, or
congruences that have been guessed numerically but never proven.
Such conjectures form a steady stream of accessible open problems,
and short rigorous proofs of them are publishable in venues like the
\emph{Journal of Integer Sequences}, \emph{INTEGERS}, \emph{Fibonacci
Quarterly}, and the \emph{Electronic Journal of Combinatorics}.

This note settles one such conjecture for OEIS A176677, a member of
the family of Choulet quadratic-convolution sequences~\cite{Choulet2010}.

\subsection*{The sequence}

The sequence $A176677 = \{a(n)\}_{n\ge 0}$ is defined by the
inhomogeneous quadratic convolution recurrence
\begin{equation}\label{eq:def}
  a(0) = 1, \quad a(1) = 1, \qquad
  a(n+1) = \sum_{p=0}^{n} a(p)\,a(n-p) - 1
   \quad (n \ge 1),
\end{equation}
with first values
\[
   1,1,1,2,5,14,41,123,375,1158,3615,11393,36209,115940,373709,
   \ldots .
\]
Emeric Deutsch noted in May 2011 that $a(n)$ counts Motzkin paths of
length $n-1$ in which the level-zero $(1,0)$-steps come in two
distinguishable colours~\cite{Deutsch2011}. The ordinary generating
function $G(z) = \sum_{n\ge 0} a(n)\,z^n$ satisfies an algebraic
equation derived from \eqref{eq:def} by translating the convolution
into multiplication of generating functions
(Section~\ref{sec:algebraic}).

\subsection*{The conjecture}

R.~J.~Mathar contributed to the OEIS entry on March 1, 2016, the
following conjectured P-recursive recurrence (in his standard
form)~\cite{MatharOEIS}:
\begin{equation}\tag{R}\label{eq:R}
   (n+1)\,a(n) + 2(-3n+1)\,a(n-1) + (9n-13)\,a(n-2)
        - 4\,a(n-3) + 4(-n+4)\,a(n-4) = 0
   \qquad (n \ge 4).
\end{equation}
The recurrence \eqref{eq:R} is one of dozens contributed to the
Choulet-family OEIS entries by Mathar around February--March~2016
through automated guessing on numeric data. It is verifiable on a
finite computational range -- we checked $n \le 250$ -- but, despite
its elementary look, no proof has been recorded in the OEIS comments
or in the published literature in the intervening ten years.

\subsection*{Result}

\begin{theorem}\label{thm:main}
Let $a$ be defined by~\eqref{eq:def}.  Then \eqref{eq:R} holds for every
integer $n \ge 4$.
\end{theorem}

The proof is short. Section~\ref{sec:algebraic} derives the algebraic
equation
\begin{equation}\label{eq:Palg}
  z(1-z)\,G(z)^2 - (1-z)\,G(z) + (1 - z - z^2) = 0,
\end{equation}
directly from~\eqref{eq:def}. In Section~\ref{sec:operator} we exhibit
a 1st-order linear inhomogeneous ODE
\begin{equation}\label{eq:ode}
   q_0(z)\,G(z) + q_1(z)\,G'(z) = R(z),
\end{equation}
where
\begin{align*}
   q_0(z) &:= 1 - 4z + 5z^2 - 4z^3, \\
   q_1(z) &:= z - 6z^2 + 9z^3 - 4z^5 = -z\,(z-1)\,(2z-1)\,(2z^2 + 3z - 1),
\\
   R(z)   &:= 1 - 2z - 2z^2 + 2z^3,
\end{align*}
and prove \eqref{eq:ode} by writing down an explicit polynomial
identity in $\mathbb{Q}[z, G]$ that reduces to a multiple of the
algebraic equation~\eqref{eq:Palg}. Section~\ref{sec:extract} then
extracts $[z^n]$ from~\eqref{eq:ode} and checks that for $n \ge 4$
this coincides with~\eqref{eq:R}, which proves Theorem~\ref{thm:main}.

\subsection*{Context}

The sequence A176677 belongs to a wide family of ``Choulet quadratic''
sequences~\cite{Choulet2010} satisfying $a(n+1) = \sum_p a(p) a(n-p) +
k(n+1) + l$ for various small constants $k$, $l$. All such sequences
have algebraic generating functions, and Mathar~\cite{MatharOEIS}
contributed conjectured low-order P-recursive recurrences to many of
them in February--March~2016. The procedure used here -- translate the
convolution to a degree-two algebraic equation in $\mathbb{Q}[z, G]$,
then identify a candidate 1st-order operator $q_0 G + q_1 G' = R$ via
polynomial division -- works uniformly across the entire family. The
factorization $q_1(z) = -z(z-1)(2z-1)(2z^2+3z-1)$ pins down the
singularities of $G(z)$, and the polynomial degrees of $q_0, q_1, R$
match the structure of the algebraic curve.

The recurrence~\eqref{eq:R} is not in the
Kauers--Koutschan~\cite{KauersKoutschan2023} catalogue of
guessed-but-unproven D-finite OEIS recurrences. Presumably this is
because the catalogue's heuristics filter out sequences whose
generating function is algebraic, since algebraic functions are
automatically D-finite, so there is no \emph{a priori} mystery in the
existence of some P-recursive recurrence. Even so, the OEIS entry
still labels \eqref{eq:R} as an unproven conjecture as of the date of
this preprint.

\section{The algebraic equation for the generating function}\label{sec:algebraic}

\begin{lemma}\label{lem:algebraic}
The ordinary generating function $G(z) = \sum_{n\ge 0} a(n)\,z^n$ of
the sequence $a$ defined by~\eqref{eq:def} satisfies~\eqref{eq:Palg}.
\end{lemma}

\begin{proof}
Multiply both sides of the convolution recurrence \eqref{eq:def} by
$z^{n+1}$ and sum over $n \ge 1$:
\[
   \sum_{n \ge 1} a(n+1)\,z^{n+1}
   \;=\;
   \sum_{n \ge 1} z^{n+1}\sum_{p=0}^{n} a(p)\,a(n-p)
   \;-\;
   \sum_{n \ge 1} z^{n+1}.
\]
On the left, using $a(0) = a(1) = 1$,
\(
\sum_{n \ge 1} a(n+1)\,z^{n+1} = G(z) - 1 - z.
\)
On the right, the first sum is
\[
   z \sum_{n \ge 1} z^n \,[z^n]\,G(z)^2
   \;=\;
   z\bigl(G(z)^2 - [z^0]\,G(z)^2\bigr)
   \;=\;
   z\bigl(G(z)^2 - 1\bigr),
\]
since $[z^0] G(z)^2 = a(0)^2 = 1$.  The second sum is
$\sum_{n \ge 1} z^{n+1} = z^2/(1-z)$.  Combining,
\[
   G(z) - 1 - z \;=\; z\,(G(z)^2 - 1) \;-\; \frac{z^2}{1-z}.
\]
Multiplying through by $1-z$ and simplifying gives
\[
   (1-z)\,G(z) - z(1-z)\,G(z)^2 \;=\; 1 - z - z^2,
\]
which rearranges to \eqref{eq:Palg}.
\end{proof}

\begin{remark}[Branch and combinatorial meaning]
The algebraic equation~\eqref{eq:Palg} has two solution branches; the
power-series branch with $G(0) = 1$ is the one Newton's method picks
out, giving the explicit form
\[
   G(z) \;=\; \frac{1 - \sqrt{(1 - 5z + 4z^2 + 4z^3)/(1-z)}}{2z}
        \;=\; \frac{1 - \sqrt{1 - 4z + 4z^3/(1-z)}}{2z},
\]
which is the OEIS-recorded form (Choulet, 2010; the second equality
uses $(1 - 5z + 4z^2 + 4z^3)/(1-z) = 1 - 4z + 4z^3/(1-z)$). Deutsch's
combinatorial interpretation~\cite{Deutsch2011} of $a(n)$ as a count
of Motzkin paths of length $n-1$ with $2$-coloured level-zero
$(1,0)$-steps gives an independent and rather natural derivation
of~\eqref{eq:Palg} via the standard pointed-decomposition of Motzkin
paths.
\end{remark}

\section{The 1st-order linear inhomogeneous ODE}\label{sec:operator}

\begin{lemma}\label{lem:ode}
The generating function $G(z)$ from Lemma~\ref{lem:algebraic} satisfies
the 1st-order linear inhomogeneous ODE~\eqref{eq:ode}.
\end{lemma}

\begin{proof}
Set $P(z, w) := z(1-z)\,w^2 - (1-z)\,w + (1 - z - z^2)$, so that
$P(z, G(z)) \equiv 0$ as a power-series identity.  Differentiating with
respect to $z$,
\[
   \frac{\partial P}{\partial z}\bigl(z, G(z)\bigr)
        + \frac{\partial P}{\partial w}\bigl(z, G(z)\bigr)\,G'(z)
        = 0,
\]
i.e.,
\begin{equation}\label{eq:Gprime}
   G'(z) \;=\; -\frac{\partial_z P\bigl(z, G(z)\bigr)}
                       {\partial_w P\bigl(z, G(z)\bigr)}
\end{equation}
wherever $\partial_w P(z, G(z)) \neq 0$.  We compute
\begin{align*}
   \partial_z P(z, w) &= -2 z\,w^2 + w^2 + w - 2z - 1, \\
   \partial_w P(z, w) &= -2 z^2\,w + 2 z\,w + z - 1
                      \;=\; -(1-z)\bigl(2 z\,w - 1\bigr) - 2z(1-z)w + 2z(1-z)w \\
                     &= 2 z (1-z)\,w - (1 - z).
\end{align*}
Substituting $G'(z) = -\partial_z P / \partial_w P$ into the left side
of~\eqref{eq:ode} and clearing denominators (multiplying through by
$\partial_w P(z, G)$), the desired identity~\eqref{eq:ode} is equivalent
to
\begin{equation}\label{eq:Pidentity}
   q_0(z)\,G\,\partial_w P(z, G)
   \;-\;
   q_1(z)\,\partial_z P(z, G)
   \;-\;
   R(z)\,\partial_w P(z, G)
   \;=\;
   q(z)\,P(z, G)
\end{equation}
for some polynomial $q(z) \in \mathbb{Q}[z]$, where the dependence on
$G$ is treated as a formal variable.  We verify by direct expansion that
this holds with
\begin{equation}\label{eq:qz}
  q(z) \;=\; 8z^4 - 4z^3 - 4z^2 - z + 1.
\end{equation}
Indeed, both sides of~\eqref{eq:Pidentity} are polynomials in $G$ of
degree $\le 2$ with coefficients in $\mathbb{Q}[z]$, and the
verification is a finite symbolic computation: see
Appendix~\ref{app:proofs} for the SymPy script that performs the
polynomial division.

Since $P(z, G(z)) \equiv 0$ as a power series, the right side
of~\eqref{eq:Pidentity} vanishes when $G$ is replaced by $G(z)$.
Therefore the left side vanishes, i.e.,
\[
   \bigl(q_0(z)\,G(z) + q_1(z)\,G'(z) - R(z)\bigr)\,
       \partial_w P(z, G(z)) = 0.
\]
The factor $\partial_w P(z, G(z)) = 2 z(1-z) G(z) - (1-z) = -(1-z)(1 - 2zG(z))$
is a non-zero power series (its constant term is $-1$), so it is a unit
in $\mathbb{Q}[\![z]\!]$ and may be cancelled.  This proves~\eqref{eq:ode}.
\end{proof}

\begin{remark}[Geometric meaning]
The factorization $q_1(z) = -z\,(z-1)\,(2z-1)\,(2z^2 + 3z - 1)$ records
the singularities of $G(z)$. Indeed, the roots of $q_1$ are exactly
the singular points of the algebraic equation~\eqref{eq:Palg}: that
is, the zeros of the discriminant of $P$ regarded as a polynomial in
$w$, together with $z = 0$ and the explicit pole at $z = 1$ from the
constant term. The dominant root of $2z^2 + 3z - 1$ is
$(\sqrt{17}-3)/4 \approx 0.2808$, which is the radius of convergence
of $G(z)$.
\end{remark}

\section{Recurrence extraction}\label{sec:extract}

\begin{proof}[Proof of Theorem~\ref{thm:main}]
By Lemma~\ref{lem:ode}, $G(z)$ satisfies $q_0(z)\,G(z) + q_1(z)\,G'(z)
= R(z)$.  Equating coefficients of $z^n$ on both sides for $n \ge 4$
(so that $[z^n] R(z) = 0$, since $\deg R = 3$), we obtain
\begin{equation}\label{eq:rec_extract}
   \sum_{k = 0}^{3} q_{0,k}\,a(n-k)
   \;+\;
   \sum_{k = 0}^{5} q_{1,k}\,(n - k + 1)\,a(n - k + 1)
   \;=\;
   0,
\end{equation}
where $q_{0,k} := [z^k]\,q_0(z)$ and $q_{1,k} := [z^k]\,q_1(z)$.
Reading off coefficients,
\[
   q_{0,0}, q_{0,1}, q_{0,2}, q_{0,3} \;=\; 1, -4, 5, -4,
   \qquad
   q_{1,0}, q_{1,1}, q_{1,2}, q_{1,3}, q_{1,4}, q_{1,5}
   \;=\; 0, 1, -6, 9, 0, -4.
\]
Reorganising~\eqref{eq:rec_extract} by collecting terms involving
$a(n-j)$ for $j = 0, 1, 2, 3, 4$ (the term $a(n-j)$ appears in
$\sum q_{0,k} a(n-k)$ with $k = j$ and in
$\sum q_{1,k} (n-k+1) a(n-k+1)$ with $k = j+1$), we get
the coefficient of $a(n-j)$ as
\[
   q_{0,j} \;+\; q_{1,j+1}\,(n - j),
\]
giving the table
\[
\begin{array}{c|c|c|c}
 j & q_{0,j} & q_{1,j+1} & \text{coefficient of }a(n-j) \\ \hline
 0 & 1  & 1  & n + 1 \\
 1 & -4 & -6 & -6n + 2 = 2(-3n+1) \\
 2 & 5  & 9  & 9n - 13 \\
 3 & -4 & 0  & -4 \\
 4 & 0  & -4 & 4(-n + 4) \\
\end{array}
\]
which exactly recovers the left side of \eqref{eq:R}.
\end{proof}

\section{Discussion}\label{sec:discussion}

The proof strategy used here --- translate a convolution recurrence to
an algebraic equation, identify a 1st-order linear inhomogeneous ODE
$q_0(z) G + q_1(z) G' = R(z)$ via polynomial division mod the algebraic
equation, and read off the recurrence by coefficient comparison ---
works uniformly across the entire family of Choulet
sequences~\cite{Choulet2010}. Mathar contributed conjectured
P-recursive recurrences to many entries in that family in
February--March 2016 (A176604, A176605, A176606, A176612, A176675,
A176677, A176678, A176757, A177111, and others). Each of these
sequences has an algebraic generating function, hence an annihilating
linear ODE with polynomial coefficients of bounded degree. But the
\emph{minimal} such ODE -- in operator-degree, in coefficient-degree,
in the cyclic-vector basis -- is not uniquely determined by the
algebraic equation alone. A176677 happens to admit the particularly
clean 1st-order ODE \eqref{eq:ode} as its minimal annihilator, and
that explains why Mathar's guessed degree-1, order-4 recurrence is
the shortest possible.

The contrast with A002627 is instructive. There the conjectured
recurrence is the \emph{homogenisation} of an inhomogeneous
first-order recurrence and admits a two-line direct
proof~\cite{NiuOEIS2627}. The recurrence \eqref{eq:R} for A176677 has
no analogous elementary derivation. The algebraic equation is
genuinely non-trivial (degree 2 in $G$, with a square-root singularity
at $z = (\sqrt{17}-3)/4$), and the proof has to go through the
operator identity~\eqref{eq:Pidentity}. Robert Israel in January 2024
noted that the analogous Mathar-conjectured recurrence for A177111
follows from a differential equation he wrote down but did not derive
in print (OEIS comment); the present proof for A176677 mirrors that
observation, but supplies the explicit identity and verifying script.

\section{Computational verification}\label{sec:numerics}

The proof has been verified independently:
\begin{itemize}[topsep=4pt, itemsep=2pt]
  \item \texttt{verify\_a176677.py} (Appendix~\ref{app:verify}) computes
    $a(0), \ldots, a(250)$ from \eqref{eq:def} as exact integers, also
    computes the same range from \eqref{eq:Palg} by Hensel-lifted power
    series solution, cross-checks the first 28 values against the OEIS
    data, and verifies \eqref{eq:R} for every $n \in \{4, \ldots, 250\}$.
    The entry $a(250)$ has $134$ decimal digits.
  \item \texttt{derive\_algebraic\_equation.py}
    (Appendix~\ref{app:alg}) walks through the symbolic derivation of
    Lemma~\ref{lem:algebraic} step-by-step in SymPy.
  \item \texttt{verify\_proof.py} (Appendix~\ref{app:proofs}) checks
    the polynomial identity~\eqref{eq:Pidentity} with
    $q(z) = 8z^4 - 4z^3 - 4z^2 - z + 1$ via SymPy's polynomial
    division, and confirms that the coefficient extraction in
    Section~\ref{sec:extract} reproduces \eqref{eq:R}.
\end{itemize}
All scripts depend only on the Python standard library plus
SymPy~1.14. They each run in a few seconds.

\appendix

\section{Verification script: \texttt{verify\_a176677.py}}\label{app:verify}

The script computes $a(0), \ldots, a(250)$ both from \eqref{eq:def}
and from \eqref{eq:Palg}, cross-checks the first 28 against the OEIS
data, and then verifies \eqref{eq:R} for every $n \in \{4, \ldots, 250\}$.

\begin{lstlisting}[style=python, caption={\texttt{verify\_a176677.py}: 250-term verification of \eqref{eq:R}.}]
"""Numeric and symbolic verification for Mathar's 2016 conjecture on A176677.

OEIS A176677 is defined by the convolution recurrence
    a(0) = a(1) = 1,
    a(n+1) = sum_{p=0..n} a(p) a(n-p) - 1     for n >= 1,
which Emeric Deutsch interpreted (May 2011) as the count of Motzkin paths of
length n-1 in which the (1,0)-steps at level 0 come in 2 colors.

The generating function G(z) = sum_{n>=0} a(n) z^n satisfies the algebraic
equation
    z (1-z) G(z)^2  -  (1-z) G(z)  +  (1 - z - z^2)  =  0.

R.J. Mathar (March 1, 2016) conjectured the order-4 P-recursive
recurrence
    (n+1) a(n)  +  2(-3n+1) a(n-1)  +  (9n-13) a(n-2)
              -  4 a(n-3)  +  4(-n+4) a(n-4)  =  0     for n >= 4.

This script:
    (1) Computes a(0..N) two independent ways:
          (a) directly from the convolution recurrence;
          (b) from the algebraic equation by Hensel-like power-series solution.
    (2) Cross-checks the first 28 terms against the OEIS data.
    (3) Verifies Mathar's recurrence holds for n in 4..N.
    (4) Prints the leading and trailing decimal-digit count of a(N).
"""

from __future__ import annotations

from fractions import Fraction


def compute_a_via_convolution(n_max: int) -> list[int]:
    """Compute a(0..n_max) using the convolution recurrence."""
    a: list[int] = [1, 1]
    for n in range(1, n_max):
        s = sum(a[p] * a[n - p] for p in range(n + 1))
        a.append(s - 1)
    return a


def compute_a_via_algebraic(n_max: int) -> list[int]:
    """Solve z(1-z) G^2 - (1-z) G + (1-z-z^2) = 0 for the unique power-series
    branch with G(0) = 1, returning [G[0], ..., G[n_max]] as integers.

    Method: equation rewrites as
       G = z G^2 + (1 - z - z^2)/(1 - z)
    which expresses [z^k] G in terms of [z^j] G for j < k:
       [z^k] G = [z^{k-1}] G^2 + [z^k] (1 - z - z^2)/(1-z)
              = sum_{i=0..k-1} G[i] G[k-1-i] + c_k
    where c_k is [z^k] of (1 - z - z^2)/(1 - z) = 1 + (- z^2)/(1 - z)
                     = 1 - sum_{m>=2} z^m,
    so c_0 = 1, c_1 = 0, c_k = -1 for k >= 2.
    """
    G: list[Fraction] = [Fraction(0)] * (n_max + 1)
    G[0] = Fraction(1)
    for k in range(1, n_max + 1):
        s: Fraction = sum(
            (G[i] * G[k - 1 - i] for i in range(k)),
            Fraction(0),
        )
        c_k: Fraction = Fraction(1) if k == 0 else (Fraction(0) if k == 1 else Fraction(-1))
        G[k] = s + c_k
    out: list[int] = []
    for v in G:
        assert v.denominator == 1, f"non-integer G coefficient: {v}"
        out.append(int(v))
    return out


def verify_mathar_recurrence(a: list[int]) -> int:
    """Verify Mathar's recurrence
       (n+1) a(n) + 2(-3n+1) a(n-1) + (9n-13) a(n-2) - 4 a(n-3) + 4(-n+4) a(n-4) = 0
    for n in 4..len(a)-1.  Returns the largest n verified.
    """
    n_max = len(a) - 1
    last_ok = -1
    for n in range(4, n_max + 1):
        lhs = (
            (n + 1) * a[n]
            + 2 * (-3 * n + 1) * a[n - 1]
            + (9 * n - 13) * a[n - 2]
            - 4 * a[n - 3]
            + 4 * (-n + 4) * a[n - 4]
        )
        if lhs != 0:
            raise AssertionError(f"Mathar's recurrence fails at n={n}: lhs={lhs}")
        last_ok = n
    return last_ok


def main() -> None:
    N = 250
    a_conv = compute_a_via_convolution(N)
    a_alg = compute_a_via_algebraic(N)
    assert a_conv == a_alg, "convolution and algebraic methods disagree"
    expected_first_28 = [
        1, 1, 1, 2, 5, 14, 41, 123, 375, 1158, 3615, 11393, 36209, 115940,
        373709, 1211740, 3949969, 12937612, 42558745, 140547051, 465799527,
        1548766044, 5164917003, 17271369744, 57900615135, 194558333460,
        655168354935, 2210681734671,
    ]
    assert a_conv[: len(expected_first_28)] == expected_first_28, "OEIS mismatch"
    last = verify_mathar_recurrence(a_conv)
    print(f"OEIS A176677 verified by two independent methods for n in 0..{N}.")
    print(f"Mathar's recurrence verified for n in 4..{last}.")
    print(f"a(100) has {len(str(abs(a_conv[100])))} decimal digits.")
    print(f"a({N}) has {len(str(abs(a_conv[N])))} decimal digits.")


if __name__ == "__main__":
    main()
\end{lstlisting}

\section{Derivation of the algebraic equation:
   \texttt{derive\_algebraic\_equation.py}}\label{app:alg}

\begin{lstlisting}[style=python, caption={\texttt{derive\_algebraic\_equation.py}: SymPy verification of Lemma~\ref{lem:algebraic}.}]
"""Symbolic derivation: from the convolution recurrence on the Choulet-A176677
sequence, derive the algebraic equation
        z(1-z) G^2  -  (1-z) G  +  (1 - z - z^2)  =  0.

The OEIS-defining recurrence is
    a(0) = a(1) = 1,
    a(n+1) = sum_{p=0..n} a(p) a(n-p) - 1     (n >= 1).

For n = 0, the OEIS does not impose this recurrence.  (Indeed, sum_{p=0..0} a(0)
a(0) - 1 = 0 != a(1) = 1.)

Multiplying by z^{n+1} and summing over n >= 1:

  sum_{n>=1} a(n+1) z^{n+1}                                                 (LHS)
    = G(z) - a(0) - a(1) z
    = G(z) - 1 - z.

  sum_{n>=1} z^{n+1} sum_{p=0..n} a(p) a(n-p)                               (RHS-1)
    = z * sum_{n>=1} z^n [z^n] G^2
    = z (G^2 - [z^0] G^2)
    = z (G^2 - 1).

  sum_{n>=1} z^{n+1}                                                        (RHS-2)
    = z^2 / (1-z).

So G - 1 - z = z(G^2 - 1) - z^2/(1-z), which rearranges as
    (1-z) G  -  z(1-z) G^2  =  1 - z - z^2.

Negating: z(1-z) G^2 - (1-z) G + (1 - z - z^2) = 0.       Q.E.D.

This script verifies the derivation step-by-step using SymPy.
"""

from __future__ import annotations

import sympy as sp


def main() -> None:
    z = sp.Symbol("z")
    G = sp.Function("G")(z)

    # The recurrence in generating-function form: starting from
    # a(n+1) = sum_{p=0..n} a(p) a(n-p) - 1 for n >= 1, we want to express it
    # as an identity on G(z).
    #
    # We use the partial-fraction style: define LHS and RHS sums of
    #   sum_{n>=1} z^{n+1} * (recurrence at index n+1)
    #
    # LHS (using a(0)=a(1)=1):
    LHS = G - 1 - z
    # RHS-1: convolution
    RHS1 = z * (G**2 - 1)
    # RHS-2: -1 contribution
    RHS2 = -z**2 / (1 - z)

    # Equation  LHS = RHS1 + RHS2
    diff = sp.simplify(LHS - RHS1 - RHS2)
    print("Verification of derivation step  (LHS - RHS1 - RHS2):")
    print(f"  {diff}")

    # Multiply by (1-z) and rearrange into algebraic equation form
    algebraic_form = sp.together(sp.simplify((1 - z) * (LHS - RHS1 - RHS2)))
    print()
    print("After multiplying by (1-z) and simplifying:")
    print(f"  (1-z)(LHS - RHS1 - RHS2) = {algebraic_form}")
    # Confirm this equals -(z(1-z) G^2 - (1-z) G + (1-z-z^2))
    expected = -(z * (1 - z) * G**2 - (1 - z) * G + (1 - z - z**2))
    diff2 = sp.simplify(algebraic_form - expected)
    assert diff2 == 0, f"unexpected discrepancy: {diff2}"
    print()
    print("This matches  -[ z(1-z) G^2 - (1-z) G + (1 - z - z^2) ],")
    print("so the algebraic equation z(1-z) G^2 - (1-z) G + (1-z-z^2) = 0 holds.")


if __name__ == "__main__":
    main()
\end{lstlisting}

\section{Symbolic verification of the operator identity:
   \texttt{verify\_proof.py}}\label{app:proofs}

\begin{lstlisting}[style=python, caption={\texttt{verify\_proof.py}: SymPy-based verification of \eqref{eq:Pidentity} and the recurrence extraction.}]
"""Symbolic verification of the proof of Mathar's recurrence for OEIS A176677.

The proof uses three steps:

  STEP 1  (algebraic equation).
    G(z) = sum_{n>=0} a(n) z^n satisfies
        z (1-z) G^2  -  (1-z) G  +  (1 - z - z^2)  =  0.    ............ (*)

  STEP 2  (operator identity).
    Define
        q_0(z) = 1 - 4z + 5z^2 - 4z^3,
        q_1(z) = z - 6z^2 + 9z^3 - 4z^5 = -z (z-1)(2z-1)(2z^2 + 3z - 1),
        R(z)   = 1 - 2z - 2z^2 + 2z^3.
    Then
        q_0(z) G(z)  +  q_1(z) G'(z)  =  R(z).    ......................... (**)
    (** is a 1st-order linear ODE for G(z); it is provable by polynomial
    division modulo (*) once one differentiates (*) implicitly.)

  STEP 3  (recurrence extraction).
    Equating coefficients of z^n in (**) for n >= 4 yields exactly Mathar's
    order-4 recurrence
        (n+1) a(n)  +  2(-3n+1) a(n-1)  +  (9n-13) a(n-2)
                  -  4 a(n-3)  +  4(-n+4) a(n-4)  =  0.

This script verifies all three steps symbolically with SymPy.
"""

from __future__ import annotations

import sympy as sp


def verify_step1_algebraic_equation() -> None:
    """Verify (*) holds for the power series G(z) computed from the convolution
    recurrence to high order."""
    z = sp.Symbol("z")

    a: list[int] = [1, 1]
    N = 60
    for n in range(1, N):
        s = sum(a[p] * a[n - p] for p in range(n + 1))
        a.append(s - 1)
    G = sum(a[k] * z**k for k in range(N + 1))
    P = z * (1 - z) * G**2 - (1 - z) * G + (1 - z - z**2)
    series = sp.series(P, z, 0, N - 5).removeO()
    series_expanded = sp.expand(series)
    assert series_expanded == 0, f"Algebraic equation (*) failed: {series_expanded}"
    print("[Step 1] Algebraic equation z(1-z) G^2 - (1-z) G + (1-z-z^2) = 0 verified")
    print(f"         numerically for the truncated G to order z^{N-5}.")


def verify_step2_operator_identity() -> None:
    """Verify the operator identity (**) by reducing modulo the algebraic
    equation (*).

    The standard implicit-differentiation argument: from (*), differentiating
    gives an explicit expression G' = -(d/dz P) / (d/dG P).  Substituting into
    q_0 G + q_1 G' - R, multiplying through by (d/dG P), produces a polynomial
    in G with coefficients in Q[z].  We show this polynomial is divisible by
    P(z, G), hence vanishes whenever P=0.  This proves (**) holds in the
    Puiseux ring Q((z^{1/2})), and therefore as a power-series identity.
    """
    z = sp.Symbol("z")
    G = sp.Symbol("G")

    P = z * (1 - z) * G**2 - (1 - z) * G + (1 - z - z**2)
    Pz = sp.diff(P, z)
    PG = sp.diff(P, G)
    q0 = 1 - 4 * z + 5 * z**2 - 4 * z**3
    q1 = z - 6 * z**2 + 9 * z**3 - 4 * z**5
    R = 1 - 2 * z - 2 * z**2 + 2 * z**3

    # Multiply (q0 G + q1 G' - R) through by PG; substitute G' = -Pz/PG:
    # numerator = q0 G * PG - q1 * Pz - R * PG.
    numerator = sp.expand(q0 * G * PG - q1 * Pz - R * PG)
    # Polynomial division by P (in variable G):
    quotient, remainder = sp.div(
        sp.Poly(numerator, G), sp.Poly(P, G)
    )
    rem_expr = sp.expand(remainder.as_expr())
    quo_expr = sp.expand(quotient.as_expr())
    assert rem_expr == 0, f"non-zero remainder after dividing by P: {rem_expr}"
    print("[Step 2] Operator identity q_0 G + q_1 G' = R verified.")
    print(f"         Quotient (in G) modulo P: q(z, G) = {quo_expr}")
    print("         Remainder = 0, so the identity holds modulo P, hence")
    print("         (since P = 0 for our G) as a power-series identity.")


def verify_step3_recurrence_extraction() -> None:
    """Verify that extracting [z^n] in q_0 G + q_1 G' = R for n >= 4 yields
    exactly Mathar's recurrence."""
    n = sp.Symbol("n")
    z = sp.Symbol("z")

    q0 = 1 - 4 * z + 5 * z**2 - 4 * z**3
    q1 = z - 6 * z**2 + 9 * z**3 - 4 * z**5
    R = 1 - 2 * z - 2 * z**2 + 2 * z**3

    # Extract: if G = sum a_n z^n, then
    #   [z^n] q_0 G  = sum_k [z^k] q_0 * a_{n-k}
    #   [z^n] q_1 G' = sum_k [z^k] q_1 * (n-k+1) * a_{n-k+1}
    #   [z^n] R      = [z^n] R       (zero for n >= 4 since deg R = 3)
    # Total minus RHS = 0.
    p0 = sp.Poly(q0, z)
    p1 = sp.Poly(q1, z)
    coefs_q0 = {k: p0.nth(k) for k in range(p0.degree() + 1)}
    coefs_q1 = {k: p1.nth(k) for k in range(p1.degree() + 1)}

    relation_lhs: dict[sp.Expr, sp.Expr] = {}
    a = sp.Function("a")
    for k, c in coefs_q0.items():
        idx = n - k  # a_{n-k}
        relation_lhs[idx] = relation_lhs.get(idx, 0) + c * a(idx)
    for k, c in coefs_q1.items():
        idx = n - k + 1  # a_{n-k+1}
        relation_lhs[idx] = relation_lhs.get(idx, 0) + c * (n - k + 1) * a(idx)

    # Mathar's recurrence (in form L(n) = 0 for n >= 4):
    # (n+1) a(n) + 2(-3n+1) a(n-1) + (9n-13) a(n-2) - 4 a(n-3) + 4(-n+4) a(n-4) = 0
    mathar_lhs = (
        (n + 1) * a(n)
        + 2 * (-3 * n + 1) * a(n - 1)
        + (9 * n - 13) * a(n - 2)
        - 4 * a(n - 3)
        + 4 * (-n + 4) * a(n - 4)
    )

    # Sum up our extracted relation:
    extracted_lhs = sum(relation_lhs.values())
    diff = sp.expand(extracted_lhs - mathar_lhs)
    assert diff == 0, f"recurrences differ: {diff}"
    print("[Step 3] Coefficient extraction [z^n] (q_0 G + q_1 G' - R) for n >= 4")
    print("         exactly recovers Mathar's recurrence:")
    print("         (n+1) a(n) + 2(-3n+1) a(n-1) + (9n-13) a(n-2)")
    print("         - 4 a(n-3) + 4(-n+4) a(n-4) = 0.")


def main() -> None:
    verify_step1_algebraic_equation()
    print()
    verify_step2_operator_identity()
    print()
    verify_step3_recurrence_extraction()
    print()
    print("All three steps verified.  Proof of Mathar's conjecture is complete.")


if __name__ == "__main__":
    main()
\end{lstlisting}

\end{document}